\documentclass[a4paper]{amsart}
\usepackage{amscd}
\usepackage{amsmath,amsfonts}
\usepackage{color}
\usepackage{graphicx}
\usepackage{float}
\usepackage[utf8]{inputenc}

\newcommand{\abs}[1]{\left\vert#1\right\vert}

\newcommand{\N}{\mathbb{N}}

\newcommand{\set}[1]{\left\{#1\right\}}
\newcommand{\norm}[2]{\| #1 \|_{#2}}
\newcommand{\scalar}[2]{\langle{ #1},{#2} \rangle}
\newcommand{\lr}[1]{\left( #1\right)}

\newenvironment{frcseries}{\fontfamily{frc}\selectfont}{}

\newcommand{\E}{D}
\newcommand{\Rg}{\mathcal R}
\theoremstyle{plain}
\newtheorem{theorem}{Theorem}
\newtheorem{corollary}{Corollary}
\newtheorem{proposition}{Proposition}
\newtheorem{lemma}{Lemma}
\theoremstyle{definition}
\newtheorem{remark}{Remark}

\title[Degree of ill-posedness for composite linear problems]{The
  degree of ill-posedness of composite linear ill-posed problems with
  focus on the impact of the non-compact Hausdorff moment operator}

\author{Bernd Hofmann}
\address{Faculty of Mathematics, Chemnitz University of Technology,
  09107 Chemnitz,  Germany}
\email{}
\author{Peter Math\'e}
\address{Weierstra{\ss} Institute for Applied Analysis and
  Stochastics, Mohrenstra{\ss}e 39, 10117 Berlin,  Germany}

\begin{document}


\begin{abstract}
  We consider compact composite linear operators in Hilbert space, where the
  composition is given by some compact operator followed by some
  non-compact one possessing a non-closed range.
  Focus is on the impact of the non-compact factor on
  the overall behaviour of the decay rates of the singular values of
  the composition. Specifically, the composition of the
  compact integration operator with the non-compact Hausdorff
  moment operator is considered. We show that the singular values of
  the composition decay faster than the ones of the integration
  operator, providing a first example of this kind. However, there is
  a gap between available lower bounds for the decay rate and the
  obtained result. Therefore we conclude with a discussion.

\end{abstract}

\maketitle

\section{Introduction}
\label{sec:intro}

We consider the following composite linear ill-posed operator
equation~$A x = y$ with
\begin{equation} \label{eq:opeq}
 \begin{CD}
  A:\; @.  X @> D >> Z  @> B>> Y
  \end{CD}
\end{equation}
where $A=B \circ \E: X \to Y$ denotes the compact linear operator with infinite dimensional range~$\Rg(A)$. This forward operator $A$ is a
composition of a compact linear operator $\E: X \to Z$ with infinite dimensional range $\Rg(\E)$ and a bounded non-compact linear operator $B: Z \to Y$ with
non-closed range $\Rg(B) \not=\overline{\Rg(B)}^Y$. Here~$X,Y$ and $Z$
denote three infinite dimensional separable real Hilbert spaces. In the nomenclature of Nashed \cite{Nashed87}, the inner problem is a linear operator equation
\begin{equation} \label{eq:inner}
\E\, x\,=\,z\,,
\end{equation}
which is ill-posed of type~II due to the compactness of $\E$, whereas the outer problem
\begin{equation} \label{eq:outer}
B\, z\,=\,y
\end{equation}
is ill-posed of type~I, since $B$ is non-compact.

Operator equations with non-compact operators possessing a non-closed range are
often assumed to be less ill-posed (ill-posedness of
type~I), and we refer to M.~Z. Nashed in~\cite[p.~55]{Nashed87} who states that ``\dots an equation
involving a bounded non-compact operator with non-closed range is
`less' ill-posed than an equation with a compact operator with
infinite-dimensional range.'' For compact operator equations it is
common to measure the \emph{degree of ill-posedness} in terms of the decay rate of the
singular values, and the above composite operator~(\ref{eq:opeq}) is of this type
despite of the non-compact factor~$B$.

In our subsequent analysis we will mainly analyze and compare
  the following cases, which are of the above type and seemingly should have similar
  properties. The compact factor~$D$ is given either
  \begin{itemize}
  \item[--] as the simple integration operator
\begin{equation}\label{eq:J}
[J x](s):=\int_0^s x(t)dt\qquad(0 \le s \le 1)
\end{equation}
mapping in $L^2(0,1)$, or
\item[--] as the natural (compact) embedding
\begin{equation}
  \label{eq:embk}
  \mathcal E^{(k)}\colon H^{k}(0,1) \hookrightarrow L^{2}(0,1)
\end{equation}
from the Sobolev space $H^{k}(0,1)$ of order $k \in \N$ to $L^2(0,1)$.
 \end{itemize}
This will be composed with~$B$ being either
\begin{itemize}
\item[--] a bounded linear multiplication operator
  \begin{equation}\label{eq:multioperator}
[B^{(M)}x](t):=m(t)\,x(t) \qquad (0 \le t \le 1)
\end{equation}
with a multiplier function $m \in L^\infty(0,1)$ possessing essential
zeros, or
\item[--] the Hausdorff moment operator $B^{(H)}: Z=L^2(0,1) \to Y=\ell^2$ defined as
\begin{equation} \label{eq:Haus}
[B^{(H)}z]_j:= \int_0^1 t^{j-1}z(t)dt \qquad (j=1,2,...).
\end{equation}
\end{itemize}
The inner operators~(\ref{eq:J}) and (\ref{eq:embk}) are known to be compact, even Hilbert-Schmidt, and
the decay rates of their singular values $\sigma_i(J)$ and $\sigma_i(\mathcal{E}^{(k)})$ to zero are available.
Both the above outer operators~(\ref{eq:multioperator}) and~(\ref{eq:Haus}) are known to be non-compact with non-closed range.

The composition~$B^{(M)}\circ J$ was studied
in~\cite{Freitag05,HW05,HW09,VuGo94}. Recent studies of the Hausdorff moment problem, which goes back to
Hausdorff's paper \cite{Hausdorff23}, have been presented
in~\cite{GHHK21}. In particular, we refer to ibid. Theorem~1 and
Proposition~13, 
which yield assertions for the composition of type $B^{(H)}\circ \mathcal{E}^{(k)}$.

The question that we are going to address is the following:
  What is, in terms of the decay of the singular values $\sigma_i(B \circ  D)$ of the composite operator $B \circ  D$  from (\ref{eq:opeq}), the impact of the non-compact outer
  operator $B$?

In the case of $B:= B^{(M)}$ and $D:=J$ results are known.
For several classes of multiplier functions~$m,$ including $m(t)=t^\theta$ for all $\theta>0$,
it was seen that the singular values of the composite operator $A$ obey the equivalence~
\begin{equation}
   \label{eq:SVDcon}
\sigma_{i}(A)  = \sigma_{i}(B^{(M)}\circ J)\; \asymp
\footnote{We shall measure the decay rates of the singular values
    asymptotically; thus for decreasing sequences~$s_{i}\geq 0$
    and~$t_{i}\geq 0$  we say that~$s_{i}\asymp t_{i}\;$ as $\;i\to\infty$
    if there are constants $0<\underline c \le \overline c < \infty$
    such that the inequalities
    $$     \underline{c} \,s_{i} \le t_{j} \le \overline{c} \,s_{i} \qquad (i=1,2,...)$$
are valid.}\; \sigma_{i}(J)\asymp
\frac 1 i\quad {\rm as} \; i\to\infty,
\end{equation}
which means that $B^{(M)}$ does not `destroy' the degree of ill-posedness of $J$ by composition.

\begin{remark}
The right-hand inequalities $\sigma_i(B \circ J) \le \overline{c} \,\sigma_i(J)$, for example required in~(\ref{eq:SVDcon}), are
trivially satisfied if $B$ is bounded. We have
$\sigma_{i}(B^{(M)}\circ J) \leq C\, \sigma_{i}(J)$ with~$C:=
\norm{B^{(M)}}{L^{2}(0,1) \to L^{2}(0,1)} $. Clearly, the same reasoning applies to
the composition operator $B^{(H)} \circ J$, and we have with $C:=
\norm{B^{(H)}}{L^{2}(0,1) \to \ell^2}= \sqrt{\pi}$
(cf.~\cite{Inglese92})
the upper estimate $\sigma_{i}(B^{(H)} \circ J) \leq
\sqrt{\pi}\,\sigma_{i}(J)\;(i=1,2,...)$.
\end{remark}

To the best of our knowledge, by now no examples are known that show a
violation of $\sigma_i(B \circ \E) \asymp \sigma_i(\E)$.
In the present study we shall show that~$\sigma_{i}(B^{(H)}\circ J)/\sigma_{i}(J)
\leq C \,i^{-1/2} \;(i=1,2,...)$ with some positive constant $C$, and the non-compact Hausdorff moment
operator~$B^{(H)}$ enlarges the degree of ill-posedness of $J$ by a
factor~$1/2$, at least.

We shall start  in
Section~\ref{sec:general} with some results for general operators,
relating conditional stability estimates to the decay of the singular
numbers of the composition~$B \circ D$.
Conditional stability estimates for the composition with the
Hausdorff moment operator are given in Section~\ref{sec:HMP}
, both for the embedding operator and the integration
operator. According to Theorem~\ref{thm:general} we
derive lower bounds for the decay rates of the
compositions~$B^{(H)}\circ \mathcal E^{(k)}$ and~$B^{(H)} \circ J$,
respectively.

The composite operators, both,~$A^*A$ for~$A=B^{(H)}\circ J$, and
$\widetilde A^* \widetilde A$ for~$\widetilde A=B^{(M)}\circ J$ are Hilbert-Schmidt operators, because the
factor~$J$ is such. In particular these may be expressed as linear Fredholm
integral operators acting in~$L^{2}(0,1)$ with symmetric positive kernels~$k$
and~$\widetilde k$, respectively. There are well-known results which
state that certain type of kernel smoothness yields a minimum decay
rate of the corresponding singular values of the integral
operator. Therefore,  in Section~\ref{sec:kernel} we establish the
form of the kernels~$k$ and~$\widetilde k$, and we study their
smoothness. In particular, for the composition~$B^{(H)}\circ J$ we
shall see that the known results are not applicable, whereas in
case~$B^{(M)}\circ J$ these known results are in alignment with~$\sigma_{i}(B^{(M)}\circ J) \asymp \sigma_{i}(J)\asymp \frac 1 i$.

Finally, in Section~\ref{sec:bounding-sing-numbers}, we
  improve the upper bounds for
the decay of the singular values of  the composition~$B^{(H)}\circ J$,
giving the first example that violates~$\sigma_{i}(B \circ D)\asymp
\sigma_{i}(D)$ as~$i\to\infty$ in the context of a non-compact outer
operator~$B$. This approach bounds the singular values by means of
bounds for the Hilbert-Schmidt norm of the composition~$\norm{\lr{ B^{(H)}\circ J}(I -
  Q_{n})}{HS}$, where $Q_n$ is a projection on the $n$-dimensional subspace of
adapted Legendre polynomials in $L^2(0,1)$. We continue to discuss the
obtained result in Section~\ref{sec:discussion}. An appendix completes the paper.

\section{Results for general operators} \label{sec:general}

We start with a general theorem explaining the interplay of
conditional stability estimates and upper bounds for the degree of
ill-posedness. To this end we shall use results from the theory
of~$s$-numbers,  and we refer to
the monograph~\cite[Prop.~2.11.6]{Pie87}. In particular, for a compact
operator, say~$T\colon X \to Y$ the singular
values~$\sigma_{i}(T)$ coincide with the corresponding (linear)
approximation numbers~$a_{i}(T)$, and hence the identities
\begin{equation}
  \label{eq:s-nums}
  \sigma_{i}(T) = \|T(I-P_{i-1})\|_{X \to Y} = \inf\{\|T - L\|_{X \to Y}: {\mathrm{dim}}(\mathcal{R}(L))<i\}
\end{equation}
hold for all $i=1,2,\dots\,$.
Above, we denote by $\{\sigma_i(T),u_i,v_i\}_{i=1}^\infty$ with
$T u_i=\sigma_i(T)v_i,\ (i=1,2,...)$ the well-defined (monotonic) singular system of
the compact operator $T$, and~$P_n: X \to X\;(n=1,2,...)$
the orthogonal projection onto~${\rm span}(u_1,...,u_n)$, the $n$-dimensional subspace of $X$, where we assign~$P_0=0: X \to X$.

The main estimate is as follows:
\begin{theorem} \label{thm:general}
Let~$\E: X \to Z$ and~$A: X \to Y$ be compact linear operators between the infinite dimensional Hilbert spaces $X,\,Y$ and
$Z$ with non-closed ranges $\mathcal{R}(\E)$ and
$\mathcal{R}(A)$. Suppose that there exists an index function $\Psi:(0,\infty) \to (0,\infty)$
such that for $0<\delta \le \|A\|_{X \to Y}$ the conditional stability estimate
\begin{equation} \label{eq:condstab}
  \sup\{\,\|\E x\|_Z:\,\|A x\|_Y \le \delta,\;\|x\|_X \le 1\}  \le \Psi(\delta)
\end{equation}
holds. Then we have
\begin{equation} \label{eq:eststab}
  \sigma_i(\E) \le \Psi(\sigma_i(A)) \qquad (i=1,2,...)
\end{equation}
and also
\begin{equation} \label{eq:estinv}
  \Psi^{-1}(\sigma_i(\E)) \le \sigma_i(A) \qquad (i=1,2,...).
\end{equation}

If the operators~$\E^{\ast}\E\colon X \to X$ and~$A^{\ast}A\colon X
\to X$ commute, and if the index function~$t\mapsto \Psi^{2}(\sqrt t),\ t>0$ is concave then the
converse holds true in the sense that~(\ref{eq:eststab}) implies the stability estimate~(\ref{eq:condstab}).
\end{theorem}
\begin{proof}
  Suppose that~(\ref{eq:condstab}) holds true. Then for every $u \in
  X,\ \|u\|_X \le 1$, we see that
\begin{equation} \label{eq:impli}
\|Au\|_Y \le \delta \quad \mbox{implies that} \quad  \|\E u\|_Z \le \Psi(\delta) \qquad (\delta>0).
\end{equation}
Consider the singular projections~$P_{i}$ for the operator~$A$.
For arbitrarily chosen~$x \in X$ with $\|x\|_X \le 1$  we see that
$$
\|A(I-P_{i-1})x\|_Y \le \|A(I-P_{i-1})\|_{X \to Y}\,\|x\|_X
\le \sigma_i(A).
$$
Applying~(\ref{eq:impli}) with $u:=(I-P_{i-1})x$ and $\delta:=\sigma_i(A)$ yields
$\|\E(I-P_{i-1})x\|_Z \le \Psi(\sigma_i(A))$. Since  $x \in X$ with
$\|x\|_X \le 1$ was chosen arbitrarily, we even arrive at
$\|\E(I-P_{i-1})\|_{X \to Y} \le \Psi(\sigma_i(A)$.

By virtue of (\ref{eq:s-nums}) we find for
$$\sigma_i(\E) =  \inf\{\|\E - L\|_{X \to Z}: {\mathrm{dim}}(\mathcal{R}(L))<i\}$$
that
\begin{equation} \label{eq:Qj}
\sigma_i(\E) \leq \|\E(I-P_{i-1})\|_{X \to Z} \le \Psi(\sigma_i(A),
\end{equation}
which proves~(\ref{eq:eststab}). Since the inverse of an index function exists and is also an index function, hence
monotonically increasing, the estimate~(\ref{eq:estinv}) is a
consequence of~(\ref{eq:eststab}).

Next, suppose that the operators~$\E^{\ast}\E$ and~$A^{\ast}A$
commute, and hence they share the same singular
functions~$u_{1},u_{2},\dots$ Clearly, for~$x=0$ we have
that~$\norm{\E x}{Z}=0\leq \Psi(\delta)$, so we may and do assume that~$x\neq 0$.
Assume that~(\ref{eq:eststab}) holds. We abbreviate~$f(t):=
\Psi^{2}(\sqrt t),\ t>0$.
First, if~$\norm{x}{X}=1$ then we bound
\begin{align*}
\norm{\E x}{Z}^{2} & = \sum_{i=1}^{\infty}
                     \sigma_{i}^{2}(\E)\abs{\scalar{x}{u_{i}}}^{2} \leq \sum_{i=1}^{\infty}
                     f\lr{\sigma_{i}^{2}(A)}\abs{\scalar{x}{u_{i}}}^{2}  \\
  & \leq
    f\lr{\sum_{i=1}^{\infty}\sigma_{i}^{2}(A)\abs{\scalar{x}{u_{i}}}^{2}}
    = f\lr{\norm{Ax}{Y}^{2}},
\end{align*}
where we used Jensen's Inequality for~$f$. Hence~$\norm{\E x}{Z} \leq
\Psi\lr{\norm{Ax}{Y}}$. Consequently, for~$x\in X, \norm{x}{X}>0$ this
extends to
\begin{equation}
  \label{eq:square-bound-normneq1}
  \frac{\norm{\E x}{Z}}{\norm{x}{X}}\leq
  \Psi\lr{\frac{\norm{Ax}{Y}}{\norm{x}{X}}},\quad x\neq 0.
\end{equation}
For the concave index function~$f$ we see that~$f(at) \geq a f(t),\ t>0$
whenever~$a\leq 1$. Thus for~$a:= \norm{x}{X}^{2}\leq 1$ and~$t:= \lr{\frac{\norm{Ax}{Y}}{\norm{x}{X}}}^{2}$ we find that
$$
\norm{\E x}{Z} \leq \Psi(\norm{Ax}{Y}),\quad x\neq 0,
$$
which in turn yields the validity of~(\ref{eq:condstab}), and this
completes the proof.
\end{proof}

\begin{remark} \label{rem:sufficient}
If the conditional stability estimate \eqref{eq:condstab} is not valid for all $\delta>0$, but for sufficiently small $\delta>0$, then the estimates  \eqref{eq:eststab} and  \eqref{eq:estinv} are
not valid for all $i \in \mathbb{N}$, but for $i$ sufficiently large. Hence, the corresponding assertions about the singular value asymptotics do not change.
\end{remark}

\begin{remark} \label{rem:modul}
We mention here that the term $$\sup\{\,\|\E x\|_Z:\,\|A x\|_Y \le \delta,\;\|x\|_X \le 1\},$$ occurring in formula \eqref{eq:condstab}, is a special case of the modulus of continuity
\begin{equation} \label{eq:modul}
\omega_M(\delta):= \sup\{\,\|\E x\|_Z:\,\|A x\|_Y \le \delta,\;x \in M\}
\end{equation}
with some closed and bounded set $M \subset X$ such that $\E M$ represents a compact set of $Z$. This is due to the compactness of the operator $\E: X \to Z$.
Note that $\omega_M(\delta)$ is increasing in $\delta>0$ with the limit condition
$\lim_{\delta \to 0\,} \omega_M(\delta)=0$. Moreover, we have for constants $E>1$ and
centrally symmetric and convex sets $M$ that $\omega_{E M}(\delta)=E\,\omega_M(\delta/E)$.
For further details of this concept we refer, for example, to \cite{BHM13,HMS08}.
In general, one is interested in bounding the modulus of continuity by a majorant index function $\Psi$ as in formula  \eqref{eq:condstab},
which leads to conditional stability estimates. Precisely in  \eqref{eq:condstab} we have the situation of a centrally symmetric and convex
$M=\{x \in X: \|x\|_X \le 1\}$ under consideration with associated majorant index function $\Psi$. Consequently, this also yields for $E>1$
$$
\sup\{\,\|\E x\|_Z:\,\|A x\|_Y \le \delta,\;\|x\|_X \le E\}  \le E\,\Psi(\delta/E).
$$
It is known from approximation theory, and it was highlighted
in~\cite[Prop.~2.9]{BHM13},  that there is always a concave
majorant for the modulus of continuity, such that without loss of
generality we may assume~$\Psi$ to be concave.
\end{remark}

\begin{remark}
  It is shown in~\cite[Thm.~4.1]{BHM13}
  that the required concavity of the function~$t\mapsto \Psi^{2}(\sqrt
  t),\ t>0$ automatically holds true whenever~$\E^{\ast}\E$ is a function of~$A^{\ast}A$,
  i.e.,\ $\E^{\ast}\E = \varphi(A^{\ast}A)$ for an index function~$\varphi$. For power type functions~$\Psi$ the
  concavity assertion  holds true if and only if~$\Psi$ is concave, see the end of
  Remark~\ref{rem:modul}.
\end{remark}



\section{Compositions with the Hausdorff moment operator} \label{sec:HMP}

In order to apply Theorem~\ref{thm:general} to compositions with  the
integration operator~$B^{(H)}$ from~(\ref{eq:Haus}), we formulate
appropriate conditional stability estimates.

\begin{theorem} \label{thm:Hausdorff}
There are constants~$C_k>0$ depending on $k=0,1,2,...$ such that
  \begin{enumerate}
  \item[(a)] \label{it:stability-E}
  For the composite problem~$B^{(H)}\circ \mathcal E^{(k)}$ the bound
  \begin{equation} \label{eq:supH}
    \sup\{\,\|x\|_{L^2(0,1)}:\,\|B^{(H)}(\mathcal E^{(k)} x)\|_{\ell^2} \le
    \delta,\;\|x\|_{H^k(0,1)} \le 1\}
    \le  \frac{C_k}{(\ln(1/\delta))^k}
  \end{equation}
  holds for sufficiently small $\delta>0$.
\item[(b)]\label{it:stability-J} For the composite problem~$B^{(H)}\circ J$ the bound
\begin{equation} \label{eq:supJ}
\sup\{\,\|J x\|_{L^2(0,1)}:\,\|B^{(H)}(J x)\|_{\ell^2} \le \delta,\;\|x\|_{L^2(0,1)} \le 1\}  \le \frac{C_0}{\ln(1/\delta)}
    \end{equation}
holds for sufficiently small $\delta>0$.
  \end{enumerate}
\end{theorem}

The proof will be along the lines of~\cite{Tal87}, and we shall state
the key points here. The analysis will be based on the (normalized) shifted
Legendre polynomials~ $\{L_j\}_{j=1}^\infty$ with the explicit representation
\begin{equation}
  \label{eq:legendre}
L_j(t)=\frac{\sqrt{2j-1}}{(j-1)!}\left(\frac{d}{dt} \right)^{j-1} t^{j-1}(1-t)^{j-1} \qquad (t \in [0,1],\;j=1,2,...)
\end{equation}
The system $\{L_j\}_{j=1}^\infty$ is the result of the Gram-Schmidt
orthonormalization process of the system $\{t^{j-1}\}_{j=1}^\infty$ of
monomials. Consequently, we have
\begin{equation}
  \label{eq:span-span}
{\operatorname{span}}(1,t,...,t^{N-1})={\operatorname{span}}(L_1,L_2,...,L_N).
\end{equation}
These polynomials form an orthonormal basis in~$L^{2}(0,1)$, and we denote~$Q_{n}$ the orthogonal
projections onto the span~$\mathcal D(Q_{n})\subset L^2(0,1)$ of the
first~$n$ Legendre polynomials, and~$P_{n}$ the projection onto the
first~$n$ unit basis vectors in~$\ell^{2}$.
\begin{lemma}\label{lem:PAQ}
  For the Hausdorff moment operator~$B=B^{(H)}$ from~(\ref{eq:Haus}) the following holds true.
  \begin{enumerate}
  \item[(I)] $P_{n} B Q_{n} = P_{n}B$,
  \item[(II)] $P_{n} B B^{\ast}P_{n} = H_{n}$
    with~$H_{n}\colon \ell^{2}_{n}\to \ell^{2}_{n}$ being the
    $n$-dimensional segment of the Hilbert matrix,
  \item[(III)]
    $\norm{Q_{n}x}{L^{2}(0,1)} \leq \frac{\norm{P_{n} B
        x}{\ell^{2}}}{\sigma_{n}(P_{n}B)}$, and
  \item[(IV)]
    $\sigma_{n}(P_{n}B) = \norm{H_{n}^{-1}}{\ell^{2}_{n}\to
      \ell^{2}_{n}}^{1/2}$.
  \end{enumerate}
  Consequently we have
  that~$\norm{Q_{n}x}{L^{2}(0,1)}
  \leq{\norm{H_{n}^{-1}}{\ell^{2}_{n}\to
      \ell^{2}_{n}}^{1/2}}{\norm{P_{n} B x}{\ell^{2}}}$.
\end{lemma}
\begin{proof}
  The first assertion (I) is easily checked and it results from the fact
  that the Gram-Schmidt matrix for turning from the monomials to the
  Legendre coefficients, see~(\ref{eq:span-span}), is lower triangular. The second assertion (II) was
  shown in~\cite[Prop.~4]{GHHK21}. The final assertion (IV) is a
  re-statement of
$$
\sigma_{n}(P_{n}B) \leq \frac{\norm{P_{n} B
    x}{\ell^{2}}}{\norm{Q_{n}x}{L^{2}(0,1)}},\quad x\neq 0.
$$
In view of the first item (I) it is enough to prove that
$$
\sigma_{n}(P_{n}B) \leq \inf_{0 \neq z\in \mathcal
  D(Q_{n})}\frac{\norm{P_{n} B z}{\ell^{2}}}{\norm{z}{L^{2}(0,1)}}.
$$
It is well known from approximation theory that
$$
\sigma_{n}(P_{n}B) = \inf_{X_{n}:\, \dim(X_{n}=n)}\;\inf_{0 \neq z\in
  \mathcal D(Q_{n})}\frac{\norm{P_{n} B
    z}{\ell^{2}}}{\norm{z}{L^{2}(0,1)}}.
 $$
 Indeed, the right-hand side above corresponds to the definition of
 the \emph{Bernstein numbers}, which constitute an $s$-number,
 see~\cite[Thm.~4.5]{Pie74}, and this proves item (III). The last
 item (IV) follows from
$$
\sigma_{n}^{2}(P_{n}B) = \sigma_{n}(P_{n} B B^{\ast}P_{n}) =
\sigma_{n}(P_{n} H P_{n}) = \sigma_{n}(H_{n}) = \frac 1
{\norm{H_{n}^{-1}}{{\ell^{2}_{n}\to
      \ell^{2}_{n}}}},
$$
which in turn yields the final assertion. The proof is complete.
\end{proof}

The next result concerns the approximation power of smooth functions
by Legendre polynomials.
\begin{lemma}\label{lem:AP-Legendre}
  For functions~$x\in H^{k}(0,1)$ there is a constant~$K_{k}$ such
  that
  \begin{equation} \label{eq:rN} \|(I - Q_{n})x\|_{L^2(0,1)} \le
    K_k\, \frac 1 {n^{k}}\quad (n\in\N).
  \end{equation}
  For~$k=1$ and hence~$x\in H^{1}(0,1)$ this my be specified as
$$
\|(I - Q_{n}) x\|_{L^2(0,1)} \le
\frac{\|x^\prime\|_{L^2(0,1)}}{2n} 
\quad (n\in\N).
$$
\end{lemma}
\begin{remark}In~\cite[Thm.~4.1]{Angbook02} the proof of~(\ref{eq:rN})
  is given for~$k=1$. In ibid.~Remark~4.1 the extension for other values of~$k$
  is stated without explicit proof. In~\cite[Thm.~2.5]{Wang12} a proof
  is given for the Legendre polynomials on the interval~$(-1,1)$,
  based on ibid.~Theorem~2.1 which describes the decay rates of the
  expansions in terms of Legendre polynomials for functions with
  Sobolev type smoothness. The specification in the second bound is
  taken from~\cite[Eq.~(27)]{Tal87}.
\end{remark}

Based on the above preparations we turn to the
\begin{proof}[Proof of Theorem~\ref{thm:Hausdorff}]
  For both assertions (a) and (b) we are going to use a decomposition of the form
  \begin{equation}
    \label{eq:z-deco}
    \norm{z}{L^{2}(0,1)} \leq \norm{Q_{n}z}{L^{2}(0,1)} +  \norm{(I - Q_{n})z}{L^{2}(0,1)}
  \end{equation}
  where~$Q_{n}$ is the orthogonal projection on the span of the first~$n$
  Legendre polynomials.

  For the first assertion (a) we let~$z:= x$, and we bound each summand.
  Recall that here~$\mathcal E^{(k)}$ is the natural embedding with $\mathcal
E^{(k)} x = x$ for all $x \in H^k(0,1)$. Thus, by Lemma~\ref{lem:PAQ} the first summand
  is bounded as
  $$
\norm{Q_{n}x}{L^{2}(0,1)} \leq \delta \norm{H_{n}^{-1}}{\ell^{2}_{n}\to
\ell^{2}_{n}}^{1/2}.
$$
From~\cite{Todd54,Wilf70} and \cite{Beckermann00} we know that there
is a constant~$\hat C$, independent of $n$, for which
$$\|H_n^{-1}\|_{\ell^{2}_{n}\to \ell^{2}_{n}} \le \hat
C\exp(4\ln(1+\sqrt{2})\,n) \le \hat C\exp(4n).
$$
This yields
\begin{equation} \label{eq:firstsum} \|Q_{n} x\|_{L^2(0,1)} \le
  \sqrt{\hat C}\exp(2n)\,\delta.
\end{equation}
The second summand in~(\ref{eq:z-deco}) is bounded in
Lemma~\ref{lem:AP-Legendre}, and altogether we find that
\begin{equation} \label{eq:total} \|x\|_{L^2(0,1)} \le \sqrt{\hat
  C}\exp(2n)\,\delta + K_k\,\frac{1}{n^{k}},
\end{equation}
We choose an integer~$N = n(\delta)$ such that the two terms on the right-hand
side of the estimate \eqref{eq:total} are equilibrated.
This is achieved by letting~$N$ be given from
$$
N = \lfloor \frac 1 4 \ln(1/\delta)\rfloor+1,\quad 0 < \delta \leq \exp(-4).
$$
Substituting $n:=N$ in~\eqref{eq:total} yields for sufficiently small
$\delta>0$
the final estimate
$$\|x\|_{L^2(0,1)} \le \frac{C_k}{\left(\ln \left( 1/\delta\right)\right)^k}, $$
with some positive constant
$C_k$ depending on $k$.

For proving the second assertion (b) we assign~$z:= Jx$. Then the  first
summand in~(\ref{eq:z-deco}) allows for an estimate of the form
$$\|Q_{n}(Jx)\|_{L^2(0,1)}\le \sqrt{\hat C} \exp(2n)\, \delta, $$
again for some constant $\hat C>0$. For bounding~$\norm{(I -
  Q_{n})(Jx)}{L^{2}(0,1)}$ we use the second estimate in
Lemma~\ref{lem:AP-Legendre} which gives, for~$\|x\|_{L^2(0,1)} \le 1$,
the bound
$$
\|(I - Q_{n}) (Jx)\|_{L^2(0,1)} \le \frac{\|x\|_{L^2(0,1)}}{2n} \leq
\frac 1 {2n}.
$$
Then we can proceed as for the first assertion
in order to
complete the proof of the second assertion, and of the theorem.
\end{proof}

The proof formulated above is an alternative to the proof of
\cite[Theorem~1]{GHHK21} for $k=1$ and an extension to the cases
$k=2,3,...$. Consequences of Theorems~\ref{thm:general} and
\ref{thm:Hausdorff} for the singular value decay rate of the Hausdorff
moment composite operator $A:=B^{(H)} \circ \mathcal{E}^{(k)}$ are
summarized in the following corollary.
\begin{corollary} \label{cor:Hausdorff} For the composite Hausdorff
  moment problem~$B^{(H)}\circ \mathcal E^{(k)}$
  there exist
  positive constants $C_k,\; \underline{C}$ and $\overline C$ such
  that
  \begin{equation*} 
    \exp(-\underline{C}\,i) \le
    \exp\left(-\left(\frac{C_k}{ \sigma_i(\mathcal
          E^{(k)})}\right)^{\frac1k} \right)
    \le \sigma_i(B^{(H)}\circ \mathcal E^{(k)}) \le \sqrt{\pi}\,\sigma_i\lr{\mathcal E^{(k)}} \le
    \frac{\overline{C}}{i^k}
  \end{equation*}
  is valid for sufficiently large indices $i \in \mathbb{N}$.
\end{corollary}
\begin{proof}
  Taking into account the well-known
  singular value asymptotics $\sigma_i\lr{\mathcal E^{(k)}} \asymp
  i^{-k}$ as~$i\to\infty$ (cf.~\cite[\S3.c]{Koenig86}) and the norm
  $\|B^{(H)}\|_{L^2(0,1)\to \ell^2}=\sqrt{\pi}$, we simply find
  for the composition~$A = B^{(H)} \circ \mathcal E^{(k)}$ the estimates from above
  $$\sigma_i(B^{(H)}\circ \mathcal E^{(k)}) \le \sqrt{\pi}\,\sigma_i\lr{\mathcal E^{(k)}}\, \le
  \frac{\overline{C}}{i^k},$$ with some positive constant
  $\overline{C}$.

  We need to show the lower bounds, and we are going to apply
  Theorem~\ref{thm:general} in combination with the estimate \eqref{eq:supH} from
  Theorem~\ref{thm:Hausdorff}.
  To do so we set 
  $X:=H^k(0,1)$, $Z:=L^2(0,1)$,
  $Y:=\ell^2$, as well as~$D:=\mathcal E^{(k)}$, $A:=B^{(H)}\circ
  \mathcal E^{(k)}$, and~$\Psi(\delta):= \frac{C_k}{(\ln(1/\delta))^k}$ for sufficiently
  small $\delta>0$. This function has the inverse~$\Psi^{-1}(t)=\exp\left(-\left(\frac{C_k}{t}\right)^{1/k}\right).$
  Then the conditional stability estimate \eqref{eq:condstab} attains
  the form~\eqref{eq:supH}, and we derive from \eqref{eq:estinv} that
$$ \exp\left(-\left(\frac{C_k}{ \sigma_i(\E)}\right)^{1/k} \right)
=\Psi^{-1}(\sigma_i(\mathcal E^{(k)})) \le \sigma_i\lr{B^{(H)}\circ
  \mathcal E^{(k)}} $$
for sufficiently large indices $i \in \mathbb{N}$. This completes the
proof.
\end{proof}

Theorem~\ref{thm:general} also applies to the composition~$B^{(H)}\circ
J$, and yields along the lines of the proof of
Corollary~\ref{cor:Hausdorff} the following result.

\begin{corollary} \label{cor:HausdorffJ}
For the composite Hausdorff moment problem~$B^{(H)}\circ J$  there exist positive constants $\underline{C}$ and $\overline C$ such that
\begin{equation} \label{eq:rough1}
\exp(-\underline{C}\,i) \le \sigma_i\lr{B^{(H)}\circ J} \le \frac{\overline{C}}{i}
\end{equation}
is valid for sufficiently large indices $i \in \N$.
\end{corollary}

 \begin{remark} \label{rem:added}
    The composite mapping~$B^{(H)}\circ J$ may be viewed as forward
    mapping when reconstructing the derivative of an unknown function
    from Hausdorff moments, captured in~$B^{(H)}$.
    The authors in~\cite{Luetal13,Zhao10} have discussed the reconstruction of an unknown function
    from Legendre moments, which will correspond to a composite
    mapping~$B^{(L)}\circ J$,
    with~$B^{(L)}\colon L^2(0,1) \to \ell^2 $, assigning to a
    function~$g\in L^2(0,1)$ the sequence of
    moments~$\int_{0}^{1} g(t) L_{i}(t)\; dt \;(i=1,2,\dots)$. In this
    case the mapping~$B^{(L)}$ constitutes a (non-compact) unitary operator, and we
    will have~$\sigma_i(B^{(L)} \circ J) = \sigma_i(J) \asymp 1/i\;(i=1,2,\dots).$
  \end{remark}

The gap between the lower and upper bounds for the singular values~$\sigma_i\lr{B^{(H)}\circ \mathcal E^{(k)}}$ and~$\sigma_i\lr{B^{(H)}\circ J}$
expressed in Corollaries~\ref{cor:Hausdorff} and~\ref{cor:HausdorffJ},
respectively, is quite large.
This gap does not allow us to decide whether the composite
  problems are \emph{moderately ill-posed} (when the upper bounds are
  realistic), or \emph{severely (exponentially) ill-posed} (when the lower bounds are
  realistic) (cf., e.g.,~\cite[Def.~8]{HofKin10}).

\section{Discussion of kernel smoothness}
\label{sec:kernel}

The composite operators that were considered so far are
Hilbert-Schmidt operators, because its factors $J$ and $\mathcal
E^{(k)}$, respectively, have this property. Hilbert-Schmidt operators
acting in~$L^{2}(0,1)$ are integral operators, and hence these can be
given in the form of a Fredholm integral operator~$[G(x)](s):=\int \limits_0^1
  k(s,t) x(t) dt\;(0 \le s \le 1)$ with kernel
  $k=k(s,t) \in L^2((0,1) \times (0,1)$.

  It is well-known that decay rates of the singular values grow with the smoothness of the kernel
  $k$, and we refer in this context to the following result.
\begin{lemma}[{see~\cite{Chang52}}] \label{lem:Allen}
  Consider in $L^2(0,1)$ the Fredholm integral operator~$[G(x)](s):=\int \limits_0^1
                                k(s,t) x(t) dt\;(0 \le s \le 1)$
 and assume that the kernel~$k$, and the derivatives~$\frac{\partial k}{\partial
    s}$,...,$\frac{\partial^{l-1}k}{\partial s^{l-1}}$ exist and are
  continuous in $s$ for almost all~$t$.
  Moreover, assume that there exist~$g \in L^2((0,1) \times (0,1))$
  and $V \in L^1(0,1)$ such that
\begin{equation} \label{eq:H1}
\frac{\partial^{l}k(s,t)}{\partial s^{l}}= \int \limits_0^s g(\tau,t)\, d\tau + V(t),
\end{equation}
Then we have
\begin{equation} \label{eq:Fredholmrates}
\sigma_i(G)=o\left( i^{-l-1.5}\right) \quad \mbox{as} \quad i \to \infty.
\end{equation}
\end{lemma}

We emphasize that Lemma~\ref{lem:Allen} provides us with upper rate bounds,
corresponding to a minimum speed of the decay to zero of the singular
values. If, in particular, the kernel is infinitely smooth on the whole unit square,
then the decay rate of the associated singular values is faster than
$\mathcal{O}(i^{-\eta})$ for arbitrarily large~$\eta>0$. Consequently
an exponential-type decay of the singular values can take place.
Lower bounds cannot be expected in general, as shows
  the simple rank-one example~$k(s,t) = (s-1/2)_{+}\times
  (t-1/2)_{+}\; (0\leq s,t\leq1)$, which exhibits low smoothness, but the
sequence of singular values with $\sigma_{1}=1$ and $\sigma_i=0\;(i=2,3,...)$ decays at any rate.
However, non-smoothness aspects like non-differentiability, non-Lipschitz and occurring poles
in the kernel give limitations for the decay rate of the singular values.
So we are not aware of examples of exponentially ill-posed
linear problems with kernel~$k$ that does not belong to~$C^\infty([0,1]\times [0,1])$.

Below, we shall determine the kernels $k$ and $\tilde k$ of the self-adjoint
companions~$A^{\ast}A$ and~$\widetilde A^{\ast}\widetilde A$
of the compositions~$A:= B^{(H)}\circ J\colon L^{2}(0,1)\to \ell^{2}$
(with the Hausdorff moment operator),
and~$\widetilde A:= B^{(M)} \circ J\colon L^{2}(0,1) \to L^{2}(0,1)$ (with
a multiplication operator), respectively.

For the first composition we have the following proposition, the proof
of which is given in the appendix.

\begin{proposition} \label{pro:kernel}
The kernel $k$ of the Fredholm integral operator $A^*A$ mapping in~$L^2(0,1)$ 
with $A=B^{(H)}\circ J$  attains the form
\begin{equation} \label{eq:kernel}
k(s,t)= \sum \limits_{j=1}^\infty \frac{(1-s^j)(1-t^j)}{j^2} \qquad (0 \le s,t \le 1).
\end{equation}
\end{proposition}

The  second composition~$\widetilde A:= B^{(M)}\circ J$ with multiplier function $m(t)=t^\theta$ for $\theta>0$ constitutes a linear Volterra
integral operator. However, it can be rewritten as a linear Fredholm integral operator
\begin{equation}\label{eq:tildeA}
[\widetilde A x](s)= \int \limits _0^1 \kappa(s,t) x(t) dt,\quad \mbox{with}  \quad \kappa(s,t)=\left\{\begin{array}{cl}
s^\theta & \; (0\leq t\leq s\leq 1)\\
0 & \; (0\leq s < t\leq 1)
\end{array}
\right..
\end{equation}
and we refer to~\cite{Freitag05,HW05} for further
investigations. Taking into account that $\kappa(t,s)$ with switched
variables is the kernel of the adjoint integral operator $\widetilde
A^*$, we have that the kernel $k$ of the operator $\widetilde A^*
\widetilde A$ mapping in $L^2(0,1)$ is given as
$$
\tilde k(s,t)=\int_0^1 \kappa(\tau,s)\kappa(\tau,t) d\tau.
$$
This yields the following proposition for the second composition case.
\begin{proposition} \label{pro:multip}
The kernel $\tilde k$ of the Fredholm integral operator $\widetilde A^* \widetilde A$ mapping in $L^2(0,1)$  with $\widetilde A$ from \eqref{eq:tildeA} attains the form
\begin{equation}\label{eq:kerneltilde}
\tilde k(s,t)=\int \limits _{\max(s,t)}^1 \tau^{2\theta}d \tau =1-\frac{\max(s,t)^{2 \theta+1}}{2\theta+1} \quad (0 \le s,t \le 1).
\end{equation}
\end{proposition}

We are going to discuss the implications of Lemma~\ref{lem:Allen} on
the decay rates of the singular values of both~$A^{\ast}A$
and~$\widetilde A^{\ast}\widetilde A$. We start with the latter.

The kernel~$\tilde k$ from \eqref{eq:kerneltilde} is continuous and satisfies for all $\theta>0$ the Lipschitz condition $\tilde k \in Lip_1([0,1] \times [0,1])$, which means that there is a constant $L>0$ such that for all $s,\hat s,t,\hat t \in [0,1]$
$$|\tilde k(s,t)-\tilde k(\hat s,\hat t)| \le L\,(|s-\hat s|+|t-\hat t|). $$
The author in~\cite{Reade83} proves that in this case we can guarantee
the decay rate
$$
\sigma_i(\widetilde A^* \widetilde A)= \mathcal{O}\left(i^{-2}\right)
\quad \mbox{as} \quad i \to \infty.
$$
The kernel $\tilde k$ from \eqref{eq:kerneltilde}, containing a maximum term,
is not differentiable at the diagonal of the unit square. If it were
continuously differentiable on  $[0,1] \times [0,1]$ then the decay
rate would even be improved to $\sigma_i(\widetilde A^* \widetilde
A)=o\left(i^{-2}\right)$, and we refer to~\cite{Read83}.
  Indeed, the exact asymptotics~$\sigma_i(\widetilde A^*
  \widetilde A) \asymp i^{-2}$ for all $\theta>0$ was shown in~\cite{HW05} .

  We turn to discussing the singular values of the
  operator~$A^{\ast}A$ with kernel $k$ from~\eqref{eq:kernel}.
  Since  the series~$\sum _{j=1}^\infty \frac{(1-s^j)(1-t^j)}{j^2}$ of continuous functions
is uniformly absolutely convergent the kernel $k$ actually belongs to the
space~$C([0,1]\times [0,1])$ of continuous functions. This allows for
partial differentiation with respect to $s$ as
\begin{equation} \label{eq:kernel1}
k_s(s,t)= \sum \limits_{j=1}^\infty \frac{-s^{j-1}(1-t^j)}{j}.
\end{equation}
Figure \ref{fig:kernels} presents a plot of the kernel $k$ and its first partial derivative $k_s$.
\begin{figure}[H]
\begin{center}
  \includegraphics[width=0.49\linewidth]{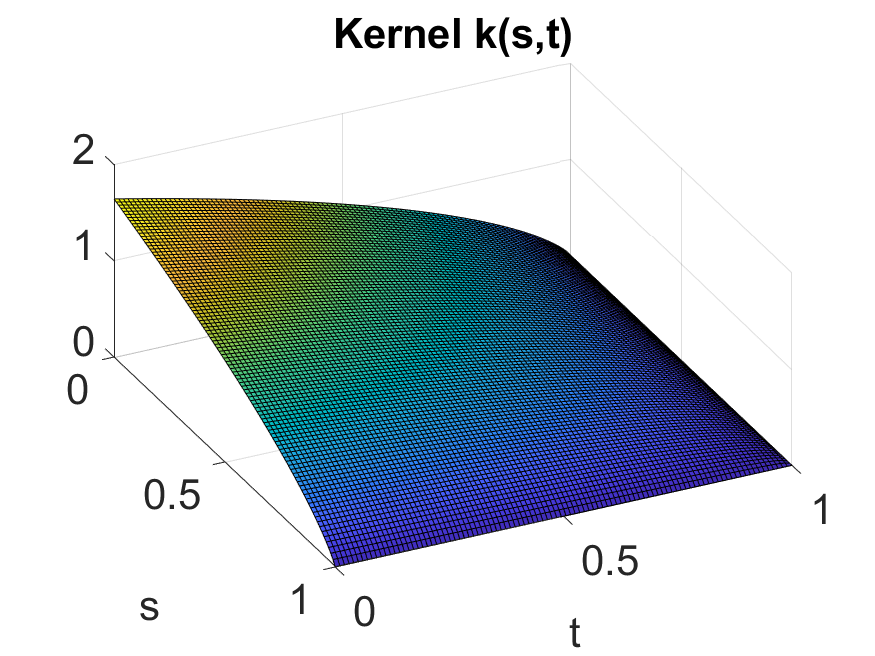}
  \includegraphics[width=0.49\linewidth]{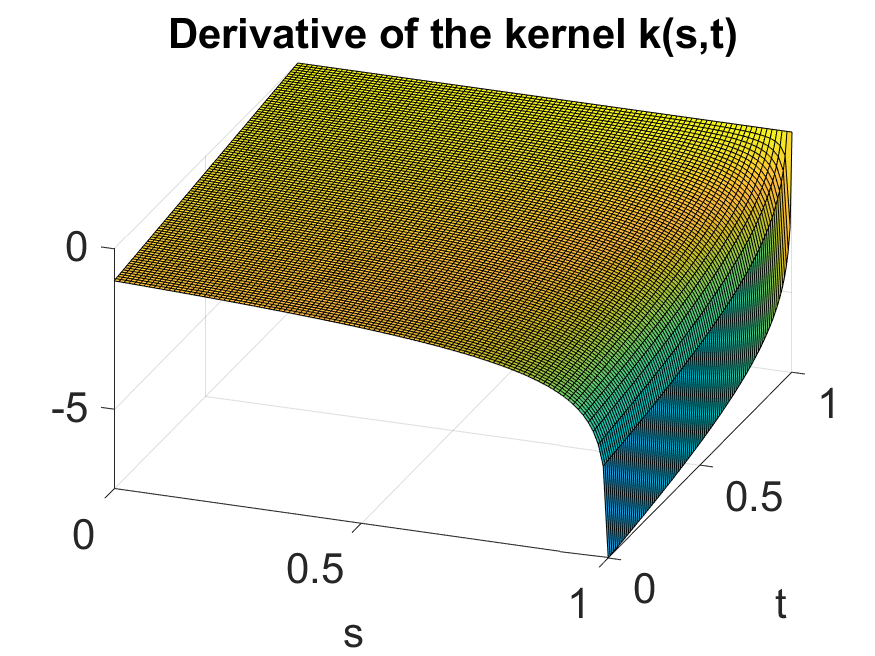}
\caption{Plot of the kernel $k$ and its derivative $k_s$} \label{fig:kernels}
\end{center}
\end{figure}
{\parindent0em The} right picture shows that the partial derivative has a pole at the right boundary with $s=1$  of the unit square. This pole implies that $k \notin Lip_1([0,1]
\times [0,1])$. On the other hand, $k_s$ is smooth elsewhere
and allows for further partial differentiation with a second partial derivative
\begin{equation} \label{eq:kernel2}
k_{ss}(s,t)= \sum \limits_{j=2}^\infty \frac{-(j-1)s^{j-2}(1-t^j)}{j}\,,
\end{equation}
which has also a pole at $s=1$.
We note that the order of the pole there is growing by one for every higher partial differentiation step
with respect to $s$.

Based on  \eqref{eq:kernel2} one can derive, in light of formula~\eqref{eq:H1} from Lemma~\ref{lem:Allen}, that
$$\frac{\partial k(s,t)}{\partial s}= \int \limits_0^s g(\tau,t)\, d\tau + V(t)$$
with $g(\tau,t)=\sum _{j=2}^\infty \frac{-(j-1)s^{j-2}(1-t^j)}{j}   $
and $V(t)=t-1$. Notice, that~$g \notin L^2((0,1) \times (0,1))$, which
prevents the application of Lemma~\ref{lem:Allen}, even  in the
case~$l=1$. Thus, Lemma~\ref{lem:Allen} is not applicable, and
  we may not make inference on the decay rates of the singular values
  by means of considering kernel smoothness.
\begin{remark}
We have not found assertions in the literature, which handle the situation
of such poles in light of decay rates of singular values.
\end{remark}

In summary, the smoothness of the kernel~$k$ from \eqref{eq:kernel} is
strongly limited. In particular we have $k \notin C^\infty([0,1]\times [0,1])$.
This makes an exponential decay rate of the singular
values~$\sigma_i(A)$ appear rather unlikely. However, at present we
have no analytical approach to check this in more detail.

\section{Bounding the singular values of the composite operator~$B^{(H)}
  \circ J$}
\label{sec:bounding-sing-numbers}

Our aim of this section is to improve the upper bound
in~\eqref{eq:rough1} for the singular values $\sigma_i(A)=\sigma_i\lr{B^{(H)}
  \circ J}$ of the composite operator
\begin{equation} \label{eq:BH-J-graph}
 \begin{CD}
 A:\;
@.  L^2(0,1) @> J >> L^2(0,1)  @> B^{(H)} >> \ell^2
  \end{CD}
\end{equation}
We emphasize that this composition constitutes a
Hilbert-Schmidt operator, since its component~$J$ is Hilbert-Schmidt, and our argument
will be based on bounding the Hilbert-Schmidt norm
$$\norm{A}{HS}= \left(\sum \limits_{i=1}^{\infty} \sigma^{2}_{i}(A) \right)^{1/2}.$$
The main result will be the following.
\begin{theorem} \label{thm:improvedrate}
For the composite Hausdorff moment problem~$B^{(H)} \circ
J$ with operators~$J$ from
\eqref{eq:J} and~$B^{(H)}$ from \eqref{eq:Haus},  there exists a positive constant $C$ such that
\begin{equation} \label{eq:rough2}
\sigma_i\lr{B^{(H)}\circ J}\le \frac{C}{i^{3/2}}\quad (i \in \N),
\end{equation}
Consequently, there is some constant~$K>0$ such that
that
$$
\sigma_{i}(B^{(H)} \circ J)/\sigma_{i}(J) \leq
\frac{K}{i^{1/2}}\quad (i \in \N).
$$
\end{theorem}

For its proof we make the following preliminary considerations.
We recall the definition of the shifted Legendre
polynomials~$L_{j}\;(j=1,2,\dots)$ from~(\ref{eq:legendre}), as well
as~$Q_{n}\colon L^{2}(0,1) \to L^{2}(0,1)$, being  the orthogonal
projection onto the $n$-dimensional subspace of the polynomials up to
degree $n-1$.

For the further estimates the next result is important. Here, we denote by~$\{\sigma_{i}(A), u_{i},v_{i})\}_{i=1}^\infty$ the singular system
of the compact operator~$A=B^{(H)} \circ J$.
\begin{proposition} \label{pro:Baumeister}
  Let~$Q_{n}$ denote the projections onto $\operatorname{span}\set{L_{1},\dots,L_{n}}$ of the Legendre polynomials up to degree $n-1$, and
  let~$P_{n}$ be the singular projection
  onto~$\operatorname{span}\set{u_{1},\dots,u_{n}}$ of the first $n$ eigenelements of $A$. Then we have
  for~$A=B^{(H)} \circ J$ that

  $$
  \sum_{i=n+1}^{\infty} \sigma^{2}_{i}(A) = \norm{A(I - P_{n})}{HS}^{2} \leq \norm{A(I - Q_{n})}{HS}^{2}.
$$
\end{proposition}
\begin{proof}
  We shall use the \emph{additivity} of the singular values,
  i.e.,\ it holds true that
  $$
  \sigma_{n+i+1}(K+L) \leq \sigma_{n+1}(K) + \sigma_{i+1}(L),\quad
  \text{for all} \ n \in \N, \ i\geq 1.
  $$
  In particular we see that
  $$
 \sigma_{n+i+1}(A) \le \sigma_{n+1}(AQ_{n}) + \sigma_{i+1}(A(I - Q_{n})) = \sigma_{i+1}(A(I - Q_{n})),
  $$
  because $\sigma_{n+1}(AQ_{n})$ vanishes by definition of $Q_n$.
  Consequently we can bound
  \begin{align*}
    \sum_{i=n+1}^{\infty} \sigma^{2}_{i}(A) &= \sum_{i=0}^{\infty}
                                      \sigma^{2}_{n+i+1}(A) \leq
                                      \sum_{i=1}^{\infty}\sigma^{2}_{i}(A(I
                                      - Q_{n}))\\
    &= \norm{A(I - Q_{n})}{HS}^{2},
  \end{align*}
  with equality for $Q_n$ being the singular projections~$P_{n}$.
\end{proof}

Finally we mention the following technical result, which is
well-known. For the sake of completeness we add a brief proof.
\begin{lemma}\label{lem:sj-kappa}
  Let~$s_{i} \;(i \in \N)$ be non-increasing, and let~$\kappa>0$. Suppose that there is a
  constant~$C_{1}<\infty$ such that~$\sum_{i=n+1}^{\infty}
  s_{i}^{2}\leq C n^{-2\kappa}$ for $n=1,2,\dots$. Then there is a
  constant~$C_{2}$ such that~$s_{i}^{2} \leq C_{2} i^{-(1+2\kappa)}$ for $i=1,2,\dots.$
\end{lemma}
\begin{proof}
We can estimate as
$$n\, s_{2n}^2 \le \sum_{i=n+1}^{2n}
  s_{i}^{2}  \leq C n^{-2\kappa},  $$
  which gives $s_{2n}^2 \le  C n^{-(1+2\kappa)}$ and proves the lemma.
\end{proof}

Let us introduce the normalized functions
$$
h_{i}(s):= \sqrt{2i+1}s^{i}\,\in L^{2}(0,1) \quad (i=0,1,2,\dots).
$$
 \begin{lemma}\label{lem:ji}
   For each~$i\geq 1,\ j\geq 2$ we have that
   $$
   \scalar{A L_{j}}{e_{i}}_{\ell^2} = -\frac{1}{i\sqrt{2i+1}}\scalar{h_{i}}{L_{j}}.
   $$
 \end{lemma}
 \begin{proof} We have $\scalar{A
     L_{j}}{e_{i}}_{\ell^2}=\scalar{B^{(H)}(J \,L_{j})}{e_{i}}_{\ell^2}
   $.
Using the formula (\ref{eq:BJ}) with~$x:= L_{j} \;\ (j=2,3,\dots)$, and since~$L_{j}
\perp 1$ for~$j\geq 2$, we see that
   $$
   \scalar{A L_{j}}{e_{i}}_{\ell^2}
   =
   -\frac{1}{i}\left[\int_{0}^{1}s^i\, L_{j}(s) \; ds \right] =
   -\frac{1}{i} \scalar{s^{i}}{L_{j}}_{L^2(0,1)} = -\frac{1}{i\sqrt{2i+1}} \scalar{h_{i}}{L_{j}}_{L^2(0,1)}.
   $$
   This completes the proof.
 \end{proof}

\begin{proof}
  [Proof of Theorem~\ref{thm:improvedrate}]
 Since the system $\{L_j\}_{j=1}^\infty$ of shifted Legendre
 polynomials is an orthogonal basis in $L^2(0,1)$, we have by virtue of~\cite[Thm.~15.5.5]{Pie78} that
 \begin{equation}
   \label{eq:HS-expand}
\norm{A (I- Q_{n})}{HS}^{2} = \sum_{j=1}^{\infty} \norm{A (I-
  Q_{n})L_{j}}{\ell^{2}}^{2}
= \sum_{j=n+1}^{\infty} \norm{A L_{j}}{\ell^{2}}^{2},
\end{equation}
and we shall bound by using Lemma~\ref{lem:ji} that
\begin{align} \label{eq:hope}
  \norm{A (I- Q_{n})}{HS}^{2} &=
    \sum_{i=n}^{\infty}\frac{1}{i^{2}(2i+1)}
    \sum_{j=n+1}^{\infty}\abs{\scalar{h_{i}}{L_{j}}_{L^{2}(0,1)}}^{2}\\
    &=  \sum_{i=n}^{\infty}\frac{1}{i^{2}(2i+1)}\norm{(I - Q_{n}) h_{i}}{L^{2}(0,1)}^{2}.\label{it:hope}
\end{align}
The norm square within the above sum is less than or equal to one,
such we arrive at
$$
 \norm{A (I- Q_{n})}{HS}^{2} \leq
 \sum_{i=n}^{\infty}\frac{1}{i^{2}(2i+1)} \leq \frac 1 2 \sum_{i=n}^{\infty}\frac{1}{i^{3}}
 $$
 The sum on the right is known to be minus one half of the second derivative
 $\psi^{(2)}(n)$ of the digamma function, see
 \cite[(6.4.10)]{Handbook64}. Thus we have
\begin{equation} \label{eq:psi1}
 \norm{A (I- Q_{n})}{HS}^{2} \le \frac{-\psi^{(2)}(n)}{4}.
\end{equation}
Moreover, from the series expansion of the digamma function, see~\cite[(6.4.13)]{Handbook64},
we see that~$\lim \limits _{n \to \infty} n^2 \,\psi^{(2)}(n)=-1$,
which implies
\begin{equation} \label{eq:psi2}
 \norm{A (I- Q_{n})}{HS}^{2} \le \frac{1}{3.999\,n^2}
\end{equation}
for sufficiently large $n$.
Finally, applying Proposition~\ref{pro:Baumeister} and
Lemma~\ref{lem:sj-kappa} (with~$\kappa=1$ and
$s_i=\sigma_i(A)\;(i \in \mathbb{N})$) 
we see that~$\sigma_i(A) \le \frac{C}{i^{1.5}}$ for some constant
$C>0$. This completes the proof.
\end{proof}

\section{Discussion}
\label{sec:discussion}

We extend the previous discussions in a few aspects. As it is seen
from Corollary~\ref{cor:HausdorffJ} and Theorem~\ref{thm:improvedrate}
there is a gap for the composition~$B^{(H)}\circ J$ between the
obtained decay rate of the order~$i^{-3/2}$ of the singular values and the
available lower bound of the order~$\exp(-\underline{C}\,i)$ as $i\to\infty$. We
shall dwell on this further, and we highlight the main points that are
responsible for the lower and upper bounds, respectively.

The overall results are entirely based on considering the Legendre
polynomials~$L_j$ as means for approximation.
Clearly, these play a prominent role in our handling of compositions that contain
the operator~$B^{(H)}$.
In particular, the normalized polynomials $L_j$ constitute an orthonormal basis in~$L^{2}(0,1)$,
and the upper bounds from Lemma~\ref{lem:AP-Legendre} show that these are
suited for approximation. However, as a consequence of using the
Legendre polynomials we arrive at the $n$-sections of the Hilbert
matrix~$H_{n}$, see Lemma~\ref{lem:PAQ}. As emphasized in the proof of
Theorem~\ref{thm:Hausdorff}, the condition numbers of the
Hilbert matrix~$H_{n}$ are of the order~$\exp(4n)$, and this in turn
yields the lower bound, after applying
Theorem~\ref{thm:general}. Despite of the fact that this general
result may not be sharp for non-commuting operators in the composition, we may argue that
using $n$-sections~$H_{n}$ is not a good advice for obtaining sharp
lower bounds. So, it may well be that the lower bounds could be
improved by using other orthonormal bases than the Legendre
polynomials.

The obtained upper bound is based on the approximation of~$B^{(H)}\circ
J$ by Legendre polynomials in the Hilbert-Schmidt norm, and we refer to the inequality
\eqref{eq:psi2}. There are indications in our analysis, for example in the context of \eqref{it:hope}, that this
bound cannot be improved, but what when using other bases?

Another aspect may be interesting. While we established an improved
rate for the composition~$B^{(H)}\circ J$, this is not possible for
the composition~$B^{(M)}\circ J$, see the discussion in
Section~\ref{sec:intro}. In the light of the spectral theorem, and we
omit details,  the
operator~$B^{(M)}$ is orthogonally  equivalent to a multiplication
operator $M_{f}$ mapping in $L^2(0,1)$ with a multiplier function~$f$ and possessing zero as
accumulation point, and isometries~$U\colon \ell^{2} \to L^{2}(0,1)$
and~$V\colon L^{2}(0,1) \to L^{2}(0,1)$,
for which we have $B^{(H)} = U^{\ast} M_{f} V$. This implies that
$$
B^{(H)} \circ J =  U^{\ast}  \circ M_{f}\circ  V \circ J.
$$
Clearly we have that~$\sigma_{i}( U^{\ast}  \circ M_{f}\circ  V \circ
J) = \sigma_{i}(M_{f}\circ  V \circ J)$, which looks very similar to
the problem of the composition $B^{(M)}\circ J$, where
$
\sigma_{i}(B^{(M)} \circ V \circ J)\asymp \sigma_{i}(B^{(M)}\circ J),
$
but with the intermediate
isometry~$V$. Therefore, we may search for isometries~$V\colon
L^{2}(0,1) \to L^{2}(0,1)$ such that we arrive at $$
\sigma_{i}(B^{(H)} \circ V \circ J)\asymp \sigma_{i}(B^{(H)}
\circ J).
$$
Clearly, this holds true for the identity, and this does
not hold true for~$V$ from above connected with the Hilbert matrix. Because isometries turn orthonormal
bases onto each other, we are again faced with the problem, which
approximating orthonormal basis is best suited as means of
approximation in the composition~$B^{(H)}\circ J$. Thus, the results
presented here are only a first step for better understanding the
problem of approximating a composition of a compact mapping followed by
a non-compact one.

\section*{Acknowledgment}
The authors express their deep gratitude to Daniel Gerth (TU Chemnitz,
Germany) for fruitful discussions and that he kindly provided Figure~\ref{fig:kernels}.
We also appreciate thanks to  Robert Plato (Univ.~of Siegen, Germany) for his hint on studying the Fredholm integral operator kernel
of $A^*A$ in $L^2$, which gives additional motivation. Bernd Hofmann is supported by the German Science Foundation (DFG) under the grant~HO~1454/13-1 (Project No.~453804957).
\appendix

\section{Proof of Proposition~\ref{pro:kernel}}
To prove the series representation~(\ref{eq:kernel}) for the kernel
$k$ 
of~$A^{\ast}A$, we start with
$$
[B^{(H)}Jx]_j=\left[\int_0^1\left(\int_0^t x(\tau)
    d\tau\right)\,t^{j-1}dt \right]_j \qquad (j=1,2,\dots,\;x \in
L^2(0,1)).
$$
Integration part parts yields
\begin{equation}\label{eq:BJ}
[B^{(H)}Jx]_j=\left[\frac{1}{j}\int_0^1 (1-t^j)\,x(t)dt \right]_j \qquad (j=1,2,\dots).
\end{equation}
Taking into account the well-known structure of the adjoint operator $(B^{(H)})^*$ of~$B^{(H)}$ (see, e.g., \cite[Proposition~3]{GHHK21}), we can further write
$$[(B^{(H)})^*B^{(H)}Jx](\eta)=\sum_{j=1}^\infty \frac{\eta^{j-1}}{j}\,\int_0^1 (1-t^j)\,x(t)dt \qquad (0 \le \eta \le 1).   $$
Applying the operator $J^*$ from the left, for which the analytical structure is also well-known, and we find again using integration by parts the formulas
$$[A^*Ax](s)=[J^*(B^{(H)})^*B^{(H)}Jx](s)=\sum_{j=1}^\infty \int_s^1\frac{\eta^{j-1}}{j}\,\left(\int_0^1 (1-t^j)\,x(t)dt\right) d\eta$$
$$=\sum_{j=1}^\infty \frac{1}{j^2}\,\int_0^1 (1-t^j)\,x(t)dt -  \sum_{j=1}^\infty \frac{s^j}{j^2}\,\int_0^1 (1-t^j)\,x(t)dt\qquad (0 \le s \le 1).$$
We can rewrite this as
$$[A^*Ax](s)=\int_0^1 \left(\sum_{j=1}^\infty
  \frac{(1-s^j)(1-t^j)}{j^2}\right) x(t)dt \qquad (0 \le s \le 1),$$
which shows the representation~(\ref{eq:kernel}),
and this completes the proof.


\end{document}